\documentclass[12pt,reqno,draft]{article} 
\usepackage{amsmath,amssymb,amsthm,amsfonts, indentfirst}
\usepackage{enumerate,color,bm}
\usepackage{mathtools}
\makeatletter
\def\@cite#1#2{[{{\bfseries #1}\if@tempswa , #2\fi}]}
\renewcommand{\section}{%
\@startsection{section}{1}{\z@}
{0.5truecm plus -1ex minus -.2ex}%
{1.0ex plus .2ex}{\bfseries\large}}
\def\@seccntformat#1{\csname the#1\endcsname.\ }
\makeatother

\numberwithin{equation}{section} 
\theoremstyle{theorem}
\newtheorem{thm}{Theorem}[section]

\newtheorem{lem}[thm]{Lemma}

\theoremstyle{definition}

\newtheorem{remark}{Remark}[section]

\newtheorem*{prth1.1}{Proof of Theorem 1.1}
\newtheorem*{prth1.1c}{Proof of Theorem 1.1 (continued)}
\newtheorem*{prcor1.2}{Proof of Corollary 1.2}
\newtheorem*{prth1.3}{Proof of Theorem 1.3}

\newcommand{\ep}{\varepsilon}
\newcommand{\pa}{\partial}
\newcommand{\Rn}{\mathbb{R}^n}
\newcommand{\R}{\mathbb{R}}

\newcommand{\cl}[1]{{\overline#1}}
\newcommand{\Tmax}{T_{\rm max}}

\begin{document}
\footnote[0]{
    2010{\it Mathematics Subject Classification}\/. 
    Primary: 35A01; Secondary: 35Q92, 92C17.
    %35A01: Existence problems: global existence, local existence, non-existence
    %35Q92: PDEs in connection with biology and other natural sciences
    %92C17: Cell movement (chemotaxis, etc.)
    }
%==========================title==========================
\begin{center} 
    \Large{{\bf 
    Global existence and boundedness 
    in a fully parabolic 
    attraction-repulsion chemotaxis system
    with signal-dependent sensitivities
    without logistic source
%    Boundedness 
%    in an attraction-repulsion chemotaxis system
%    with signal-dependent sensitivities
    }}%
\end{center}
\vspace{5pt}
%===========================author=========================
\begin{center}
    Yutaro Chiyo\\%, %\footnote[0]{
    %\\[2mm]
    \vspace{2mm}
    Department of Mathematics, 
    Tokyo University of Science\\
    1-3, Kagurazaka, Shinjuku-ku, 
    Tokyo 162-8601, Japan\\
    \vspace{6mm}

    Masaaki Mizukami\footnote{Corresponding author.}\\
    \vspace{2mm}
    Department of Mathematics, Faculty of Education, 
    Kyoto University of Education \\
    1, Fujinomori, Fukakusa, Fushimi-ku, Kyoto 
    612-8522, Japan\\
    \vspace{6mm}

    %\\[2mm]
    Tomomi Yokota%
      \footnote{Partially supported by Grant-in-Aid for
      Scientific Research (C), No.\,21K03278.}
    \footnote[0]{
    E-mail: 
    {\tt ycnewssz@gmail.com}, %}\footnote[0]{
    {\tt masaaki.mizukami.math@gmail.com}, 
    {\tt yokota@rs.tus.ac.jp}%[2mm]
    }\\
    \vspace{2mm}
    Department of Mathematics, 
    Tokyo University of Science\\
    1-3, Kagurazaka, Shinjuku-ku, 
    Tokyo 162-8601, Japan\\
%    \vspace{12pt}
    \vspace{2pt}
\end{center}
\begin{center}    
    \small \today
\end{center}

\vspace{2pt}
%=====================  Abstract  =======================
\newenvironment{summary}
{\vspace{.5\baselineskip}\begin{list}{}{%
     \setlength{\baselineskip}{0.85\baselineskip}
     \setlength{\topsep}{0pt}
     \setlength{\leftmargin}{12mm}
     \setlength{\rightmargin}{12mm}
     \setlength{\listparindent}{0mm}
     \setlength{\itemindent}{\listparindent}
     \setlength{\parsep}{0pt}
     \item\relax}}{\end{list}\vspace{.5\baselineskip}}
\begin{summary}
{\footnotesize {\bf Abstract.}
    This paper deals with the fully parabolic attraction-repulsion 
    chemotaxis system with signal-dependent sensitivities,
%
%%============  problem  ============
    \begin{align*}
        \begin{cases}
        %==========  equations  ==========
        u_t=\Delta u-\nabla \cdot (u\chi(v)\nabla v)
                         +\nabla \cdot (u\xi(w)\nabla w),
        &x \in \Omega,\ t>0,
    \\[1.05mm]
        v_t=\Delta v-v+u,
        &x \in \Omega,\ t>0,
    \\[1.05mm]
        w_t=\Delta w-w+u,
        &x \in \Omega,\ t>0
        \end{cases}
    \end{align*}
    under homogeneous Neumann boundary conditions and initial conditions, 
    where $\Omega \subset \mathbb{R}^n$ $(n \ge 2)$ is a bounded domain with
    smooth boundary, 
    $\chi, \xi$ are functions satisfying some conditions. 
    Global existence and boundedness of classical solutions 
    to the system with logistic source 
    have already been obtained 
    by taking advantage of the effect of logistic dampening 
    (J. Math.\ Anal.\ Appl.; 2020;489;124153). 
    This paper establishes existence of global bounded classical solutions 
    despite the loss of logistic dampening.
}
\end{summary}
\vspace{10pt}

\newpage
%%========================================================%%
%%==============                                                   =============%%
%%======                                  Section1                               ======%%
%%====                                                                                    ====%%
%%==                                                                                            ==%%
%%====                                   Introduction                                 ====%%
%%======                                                                              ======%%
%%==============                                                  ==============%%
%%========================================================%%

\section{Introduction} \label{Sec1}

Recently, in~\cite{CMY-2020}, we studied the fully parabolic 
attraction-repulsion chemotaxis system 
with signal-dependent sensitivities and logistic source,
% 
%=====================  problem  =======================
    \begin{align} \label{ARC2}
        \begin{cases}
        %==========  equations  ==========
        u_t=\Delta u-\nabla \cdot (u\chi(v)\nabla v)
                         +\nabla \cdot (u\xi(w)\nabla w)
                         +\mu u(1-u),
    \\[1.05mm]
        v_t=\Delta v-v+u,
    \\[1.05mm]
        w_t=\Delta w-w+u,
        \end{cases}
    \end{align}
where $\chi, \xi$ are some decreasing functions and $\mu>0$. 
In the literature we obtained global existence and boundedness 
in \eqref{ARC2} by taking advantage of the effect of the logistic term. 
In light of this result, 
the following question is raised:
\begin{center}
{\it Does boundedness of  solutions still hold 
without logistic term?}
\end{center}

In order to provide an answer to the above question, 
this paper focuses on 
the fully parabolic attraction-repulsion chemotaxis system with 
signal-dependent sensitivities, 
% 
%=====================  problem  =======================
    \begin{align} \label{ARC}
        \begin{cases}
        %==========  equations  ==========
        u_t=\Delta u-\nabla \cdot (u\chi(v)\nabla v)
                         +\nabla \cdot (u\xi(w)\nabla w),
                         %+\mu u(1-u),
        &x \in \Omega,\ t>0,
    \\[1.05mm]
        v_t=\Delta v-v+u,
        &x \in \Omega,\ t>0,
    \\[1.05mm]
        w_t=\Delta w-w+u,
        &x \in \Omega,\ t>0,
    \\[1.8mm]
        %==========  boundary conditions  ==========
        \nabla u \cdot \nu=\nabla v \cdot \nu=\nabla w \cdot \nu=0, 
        &x \in \pa \Omega,\ t>0,
    \\[1.05mm]
        %==========  initial conditions  ==========
        u(x, 0)=u_0(x),\ v(x, 0)=v_0(x),\ w(x, 0)=w_0(x),
        &x \in \Omega,
        \end{cases}
    \end{align}
where $\Omega \subset \Rn$ $(n \ge 2)$ is a bounded domain 
with smooth boundary $\pa \Omega$; 
$\nu$ is the outward normal vector to $\pa \Omega$; 
$\chi, \xi$ are positive known functions; 
%generalizing the prototypes
%\begin{align*}
% \chi(s) = \frac{a}{(b+s)^k}, \quad \xi(s) = \frac{c}{(d+s)^{\ell}} 
% \quad (a,c>0,\ b, d \ge 0,\ k,\ell>1); 
%\end{align*}
$u, v, w$ are unknown functions.  
The initial data $u_0, v_0, w_0$ are supposed to be nonnegative 
functions satisfying
% 
%================  condition for known functions  ==================
    \begin{align} \label{KF}
            &u_0 \in C^0(\cl{\Omega})
    \quad
            {\rm and}
    \quad 
            u_0 \neq 0,
\\[1.05mm] \label{KF2}
            &v_0, w_0 \in W^{1,\infty}(\Omega).
%\quad
%            {\rm and}
%    \quad 
%   v_0>0,\ w_0>0.
    \end{align}
\vspace{-4mm}

We now explain the background of 
\eqref{ARC}. 
Chemotaxis is a property of cells to move in response to 
the concentration gradient of a chemical substance produced by the cells. 
The origin of the problem of describing such biological phenomena is 
the chemotaxis model proposed by Keller--Segel~\cite{KS-1970}.
The above systems \eqref{ARC2} and \eqref{ARC} describe a process by which cells move 
in response to a chemoattractant and a chemorepellent produced 
by the cells themselves. 
There are a lot of studies for such attraction-repulsion chemotaxis systems, and 
we introduce some of them, by reducing 
parameters to $1$, as follows.
\vspace{2mm}

As to the system \eqref{ARC2} with constant sensitivities (i.e., 
$\chi(v) \equiv \chi$, $\xi(w) \equiv \xi$, where $\chi, \xi>0$ are constants), 
global existence and boundedness were obtained in~\cite{J-2015, JL-2015, JW-2015, JW-2020}. 
More precisely, Jin--Wang~\cite{JW-2015} investigated 
the one-dimensional case. 
Also, when $\chi=\xi$, Jin--Liu~\cite{JL-2015} studied 
the two- and three-dimensional cases. 
Recently, Jin--Wang~\cite{JW-2020} derived global existence and 
boundedness, and stabilization under the condition 
$\frac{\xi}{\chi} \ge C$ with some $C>0$ in the two-dimensional setting. 
In this way, boundedness is well established in the system \eqref{ARC2} 
with logistic term.

On the other hand, the system \eqref{ARC} with $\chi(v) \equiv \chi$, $\xi(w) \equiv \xi$ (positive constants) 
has been studied in \cite{L-pre, LT-2015}. 
In the two-dimensional setting, Liu--Tao~\cite{LT-2015} establishd 
global existence and boundedness 
under the condition $\chi<\xi$. 
In contrast, Lankeit~\cite{L-pre} showed finite-time blow-up 
%in the three-dimensional setting under some additional conditions 
in the three-dimensional radial setting
when $\chi>\xi$. 
Moreover, some related works which deal with
the parabolic--elliptic--elliptic version of \eqref{ARC2} 
can be found in~\cite{SS-2019, TW-2013, V-2019, YGZ-2017}. 
Specifically, Tao--Wang~\cite{TW-2013} derived 
global existence and boundedness under the condition $\chi<\xi$ in two or more space 
dimensions, 
while finite-time blow-up was proved in the two-dimensional setting 
when $\chi>\xi$ and the initial data satisfy 
%some condition. 
the conditions that $\int_\Omega u_0>\frac{8\pi}{\chi-\xi}$ and that 
$\int_\Omega u(x)|x-x_0|^2\,dx$ ($x_0 \in \Omega$) is sufficiently small. 
Salako--Shen~\cite{SS-2019} obtained 
global existence and boundedness when $\chi=0$ (or $\mu>\chi-\xi+M$ with some $M>0$ in \eqref{ARC2}). 
Whereas in the two-dimensional setting, 
finite-time blow-up was shown by Yu--Guo--Zheng~\cite{YGZ-2017} 
and lower bound of blow-up time was given by Viglialoro~\cite{V-2019} 
under the condition $\chi>\xi$.
Thus we can see that if there is no logistic term, 
boundedness breaks down in some cases.

\vspace{2mm}

As mentioned above, the logistic term seems helpful to derive  
boundedness in \eqref{ARC2}, 
whereas it is not clear whether boundedness in \eqref{ARC} without logistic term holds or not. 
The purpose of this paper is to establish 
a result on boundedness in  \eqref{ARC} with extra information about the outcome by %\cite{CMY-2020}.
our previous work.
\vspace{2mm}

We now introduce conditions for the functions $\chi, \xi$ 
in order to state the main theorem. 
We assume throughout this paper that 
$\chi, \xi$ satisfy the following conditions:
% 
%================  condition for \chi, \xi  ==================
    \begin{align}
            \label{chiclass}
            &\chi \in C^{1+\theta_1}([\eta_1,\infty)) 
                          \cap L^1(\eta_1,\infty)\ 
            (0<\exists\, \theta_1<1),
    \quad 
            \chi>0,  
    \\[3.05mm]
            \label{xiclass}
            &\xi \in C^{1+\theta_2}([\eta_2,\infty)) 
                        \cap L^1(\eta_2,\infty)\ 
            (0<\exists\, \theta_2<1),
    \quad 
            \xi>0,
    \\[3.05mm]
            \label{chibound}
            &\exists\, \chi_0>0;
    \quad 
            s\chi(s) \le \chi_0
    \quad 
            {\rm for\ all}\ s \ge \eta_1,          
    \\[3.05mm]
            \label{xibound}
            &\exists\, \xi_0>0;
    \quad 
            s\xi(s) \le \xi_0
    \quad 
            {\rm for\ all}\ s \ge \eta_2,
    \\[3.05mm]
            \label{chiineq}
            &\exists\, \alpha>0;
    \quad 
            \chi'(s)+\alpha|\chi(s)|^2 \le 0
    \quad 
            {\rm for\ all}\ s \ge \eta_1,
    \\[3.05mm]
            \label{xiineq}
            &\exists\, \beta>0;
    \quad 
            \xi'(s)+\beta|\xi(s)|^2 \le 0
    \quad 
            {\rm for\ all}\ s \ge \eta_2,
    \end{align}
where $\eta_1, \eta_2\ge0$ 
are constants which will be fixed 
in Lemma~\ref{LSE}; 
note that if $v_0, w_0>0$, 
then we can take $\eta_1, \eta_2>0$ (see also \eqref{eta} 
with $z_0>0$). 
Moreover, we suppose that $\alpha, \beta$ appearing 
in the conditions \eqref{chiineq}, \eqref{xiineq} satisfy
% 
%==========  condition for \alpha, \beta  ==========
    \begin{align} \label{condiab}
          \alpha>\frac{\frac{n}{2}(2\delta+1)\big((n-1)\delta+n\big)+\sqrt{D}}
                     {2\delta(\beta-n)-\delta^2-\frac{n}{2}},
    \quad 
          \beta>n+\sqrt{\frac{n}{2}}
    \end{align}
%    \begin{align} \label{condiab}
%           (\alpha,\beta) 
%    &\in \bigcup_{\lam'>\sqrt{\frac{n}{2}}}\ 
%           \bigcup_{\delta' \in \big(
%                                     \lam'-\sqrt{(\lam')^2-\frac{n}{2}},\ 
%                                     \lam'+\sqrt{(\lam')^2-\frac{n}{2}}\,
%                                     \big)}
%           X(n, \lam', \delta'),
%    \end{align}
 for some $\delta \in J:= \left(\beta-n-\sqrt{(\beta-n)^2-\frac{n}{2}},\, 
 \beta-n+\sqrt{(\beta-n)^2-\frac{n}{2}}\right)$, where
% 
%==========  def of D  ==========
    \begin{align*}
         D
    &:= \frac{n\delta}{2}\big(2\beta+(2n-1)\delta\big)  
         \Big[
         2\delta(\beta-n)+\frac{n}{2}\Big(2n(\delta+1)^2-(2\delta+1)^2\Big)
         \Big].
    \end{align*}

\vspace{6mm}
\noindent Then the main result reads as follows.

% 
%=====================  Theorem 1.1  =======================
\begin{thm} \label{mainthm}
 Let\/ $\Omega \subset \Rn$ $(n \ge 2)$ be a bounded domain 
 with smooth boundary. %and let $\mu>0$. 
 Assume that $(u_0, v_0, w_0)$ satisfy \eqref{KF}, \eqref{KF2}. 
 Suppose that $\chi, \xi$ fulfill \eqref{chiclass}--\/\eqref{xiineq} 
 with $\alpha, \beta$ which satisfy \eqref{condiab} 
 for some $\delta \in J$. 
 Then there exists a unique triplet $(u, v, w)$ of nonnegative functions
% 
%==========  class of solution  ==========
    \begin{align*}
        &u \in C^0(\cl{\Omega} \times [0,\infty)) 
                 \cap C^{2,1}(\cl{\Omega} \times (0,\infty)),
    \\
        &v, w \in C^0(\cl{\Omega} \times [0,\infty)) 
                 \cap C^{2,1}(\cl{\Omega} \times (0,\infty)) 
                 \cap L^\infty(0,\infty; W^{1,\infty}(\Omega)),
%        &w \in C^0(\cl{\Omega} \times [0,\infty)) 
%                 \cap C^{2,1}(\cl{\Omega} \times (0,\infty)) 
%                 \cap L^\infty(0,\infty; W^{1,\infty}(\Omega)),
    \end{align*}
 which solves \eqref{ARC} in the classical sense. 
 Moreover, the solution $(u, v, w)$ is bounded uniformly-in-time\/{\rm :}
% 
%==========  uniformly-in-time  ==========
    \begin{align*}
           \|u(\cdot, t)\|_{L^\infty(\Omega)}
         +\|v(\cdot, t)\|_{W^{1,\infty}(\Omega)}
         +\|w(\cdot, t)\|_{W^{1,\infty}(\Omega)} 
    \le C
    \end{align*}
 for all $t>0$ with some $C>0$.
\end{thm}

\begin{remark}
%The above theorem covers 
%a wider class of functions $\chi$ and $\xi$ in the previous literature. 
The condition \eqref{condiab} is identical to that in the previous literature, 
although the logistic term 
disappears.
\end{remark}
\vspace{5mm}

The strategy for the proof of Theorem \ref{mainthm} is
to show $L^p$-boundedness of $u$ 
with some $p>\frac{n}{2}$. 
The key is to derive the differential inequality
%
%%==============  DI of \int u^pf  ==============
    \begin{align} \label{DIupf}
         \frac{d}{dt} \int_\Omega u^pf(x, t)\,dx
    \le c_1\int_\Omega u^p f(x, t)\,dx-c_2
         \Big(\int_\Omega u^p f(x, t)\,dx\Big)^{\frac{1}{\theta}}+c_3
    \end{align}
with some constants $c_1, c_2, c_3>0$, $\theta\in(0,1)$ and some
function $f$ defined by using $v, w$. 
In our previous work including the logistic term $\mu u(1-u)$,  the differential inequality 
    \begin{align*}
         \frac{d}{dt} \int_\Omega u^pf(x, t)\,dx
    \le c_1\int_\Omega u^p f(x, t)\,dx-\mu p |\Omega|^{-\frac{1}{p}}
         \Big(\int_\Omega u^p f(x, t)\,dx\Big)^{1+\frac{1}{p}}
    \end{align*}
was established. 
The second term on the right-hand side of this inequality, 
which is important in proving boundedness, 
is derived from the effect of the logistic term. 
In other words, in the absence of a logistic source, 
the proof in the previous work fails and we cannot obtain boundedness.
Thus we will show the differential inequality \eqref{DIupf} with 
the help of the diffusion term rather than the logistic term.
\vspace{2.5mm}

This paper is organized as follows. 
In Section \ref{Sec2} we collect some preliminary facts about 
local existence of classical solutions to \eqref{ARC} and 
a lemma such that an $L^p$-estimate for $u$ 
with some $p>\frac{n}{2}$ implies an $L^\infty$-estimate for $u$. 
Section \ref{Sec3} is devoted to the proof of global existence 
and boundedness (Theorem \ref{mainthm}).
\vspace{5mm}

%%========================================================%%
%%==============                                                  ==============%%
%%======                              Section2                                   ======%%
%%====                                                                                     ====%%
%%==                                                                                            ==%%
%%====                               Preliminaries                                    ====%%
%%======                                                                              ======%%
%%==============                                                  ==============%%
%%========================================================%%

\section{Preliminaries} \label{Sec2}

In this section we give some lemmas which will be used later. 
We first present the result obtained by a similar argument in \cite[Lemma 2.2]{F-2015} 
(see also \cite[Lemma 2.1 and Remark 2.2]{MY-2017}),  which will be applied to the second and third equations in \eqref{ARC}. 
%
%====================  lower bound  ====================
\begin{lem}\label{LB}
Let $T>0$. 
Let $u \in C^0(\cl{\Omega} \times [0,T))$ be a nonnegative function 
such that, with some $m>0$, $\int_\Omega u(\cdot, t)=m$ for all 
$t \in [0, T)$. 
If $z_0 \in C^0(\cl{\Omega})$, 
$z_0\ge0$ in $\cl{\Omega}$ and 
$z \in 
C^0(\cl{\Omega} \times [0,T)) \cap C^{2,1}(\cl{\Omega} \times (0,T))$ 
is a classical solution of
    \begin{align*}
        \begin{cases}
        %==========  equations  ==========
        z_t=\Delta z-z+u,
        &x \in \Omega,\ t>0,
    \\[1.05mm]
        \nabla z \cdot \nu=0, 
        &x \in \pa \Omega,\ t>0,
    \\[1.05mm]
        %==========  initial conditions  ==========
        z(x, 0)=z_0(x), 
        &x \in \Omega,
        \end{cases}
    \end{align*}
then for all $t \in (0, T)$, 
%
%=========  lower bound  ==========
\begin{align*}
   &\inf_{x \in \Omega} z(x, t) \ge \eta
\end{align*}
%================================
%
with
\begin{align}\label{eta}
\eta:=\sup_{\tau>0}\Big(\min
         \Big\{e^{-2\tau}\min_{x \in \cl{\Omega}}z_0(x),\ 
         c_0m(1-e^{-\tau})\Big\}\Big)\ge0,
\end{align}
where $c_0>0$ is a lower bound for the fundamental solution of 
$\varphi_t=\Delta\varphi-\varphi$ with Neumann boundary condition. 
\end{lem}
%=======================================================
%
We next introduce a result on local existence of classical solutions 
to \eqref{ARC}.
%
%%=====================  Lemma 2.1  =======================
\begin{lem} \label{LSE}
 Let $n \ge 1$ and let $(u_0, v_0, w_0)$ fulfill \eqref{KF}, \eqref{KF2}. 
Put $m_0:=\int_\Omega u_0$ and 
let $\eta_1, \eta_2\ge0$ be 
constants given by \eqref{eta} with $(z_0, m)=(v_0, m_0)$ 
and $(z_0, m)=(w_0, m_0)$, respectively. 
Assume that  
 $\chi \in C^{1+\theta_1}([\eta_1,\infty))$, 
$\xi \in C^{1+\theta_2}([\eta_2,\infty))$ 
 with some $\theta_1, \theta_2 \in (0,1)$.
 Then there exists $\Tmax \in (0,\infty]$ such that 
 \eqref{ARC} admits a unique classical solution 
 $(u, v, w)$ such that
% 
%==========  class of solution  ==========
    \begin{align*}
        &u \in C^0(\cl{\Omega} \times [0,\Tmax)) 
                 \cap C^{2,1}(\cl{\Omega} \times (0,\Tmax)),
    \\
        &v,w \in C^0(\cl{\Omega} \times [0,\Tmax)) 
                 \cap C^{2,1}(\cl{\Omega} \times (0,\Tmax)) 
                 \cap L^\infty_\mathrm{loc}([0,\Tmax); W^{1,\infty}(\Omega))
%    \\
%        &w \in C^0(\cl{\Omega} \times [0,\Tmax)) 
%                 \cap C^{2,1}(\cl{\Omega} \times (0,\Tmax)) 
%                 \cap L^\infty_\mathrm{loc}([0,\Tmax); W^{1,\infty}(\Omega))
    \end{align*}
 and $u$ has positivity as well as the mass conservation property
\begin{align}\label{mcp}
\int_\Omega u(\cdot, t)=\int_\Omega u_0
\end{align}
for all $t \in (0, \Tmax)$,
whereas $v$ and $w$ satisfy the lower estimates
   \begin{align}
\inf_{x \in \Omega} v(x, t) \ge \eta_1,\quad
\inf_{x \in \Omega} w(x, t) \ge \eta_2\label{LEw}
\end{align}
for all $t \in (0, \Tmax)$. 
 Moreover, 
% 
%==========  extension criterion  ==========
    \begin{align} \label{BU}
            {\it if}\ \Tmax<\infty,
    \quad 
            {\it then}\ \limsup_{t \nearrow \Tmax} 
                           \left(\|u(\cdot,t)\|_{L^\infty(\Omega)}
                           +\|v(\cdot,t)\|_{W^{1, \infty}(\Omega)}
                           +\|w(\cdot,t)\|_{W^{1, \infty}(\Omega)}\right)=\infty.
    \end{align}
\end{lem}
%%==================  Proof of Lemma 2.1  ====================
\begin{proof}
 Using a standard argument based on the contraction mapping principle 
 as in~\cite[Lemma 3.1]{TW-2013}, we can show local existence and 
 blow-up criterion \eqref{BU}. 
Note that the mass conservation property \eqref{mcp} can be obtained 
 by integrating the first equation in \eqref{ARC} 
 over $\Omega \times (0, t)$ for $t\in(0, \Tmax)$, and that the lower estimates \eqref{LEw} follow from Lemma~\ref{LB}. 
\end{proof}

In the following we assume that $\Omega \subset \R^n$ 
$(n \ge 2)$ 
is a bounded domain with smooth boundary, 
$\chi, \xi$ fulfill \eqref{chiclass}, \eqref{xiclass}, respectively,  
%$\mu>0$ and 
$(u_0, v_0, w_0)$ satisfies \eqref{KF}, \eqref{KF2}. Then we
denote by $(u, v, w)$ the local classical
solution of \eqref{ARC} given in Lemma \ref{LSE} and 
by $\Tmax$ its maximal existence time. 
\vspace{2mm}

We next give the following lemma which tells 
us a strategy to prove global existence and boundedness. 
%For the proof, see~\cite[Lemma~2.3]{CMY-2020}; 
%note that it is rather easier to show without a logistic term.
%
%%=====================  Lemma 2.3  =======================
\begin{lem} \label{bddu}
 Assume that $\chi, \xi$ fulfill that $\chi(s) \le K_1$,   
 $\xi(s) \le K_2$ for all $s \ge 0$ with some $K_1, K_2>0$, respectively.  
 If there exist $K_3>0$ and $p>\frac{n}{2}$ satisfying
%
%%========== L^p boundedness (condition) ==========
    \begin{align*}
            \|u(\cdot, t)\|_{L^p(\Omega)} \le K_3
    \end{align*}
 for all $t \in (0, \Tmax)$, then we have 
% 
%%========== L^\infty boundedness (conclusion) ==========
    \begin{align*} %\label{Linftybdd}
            \|u(\cdot,t)\|_{L^\infty(\Omega)}
                           +\|v(\cdot,t)\|_{W^{1, \infty}(\Omega)}
                           +\|w(\cdot,t)\|_{W^{1, \infty}(\Omega)} \le C
    \end{align*}
 for all $t \in (0, \Tmax)$ with some $C>0$.
\end{lem}
\begin{proof} 
See~\cite[Lemma~2.3]{CMY-2020}; 
note that it is rather easier to show without logistic term.
\end{proof}
\vspace{5mm}

%%==============================================================%%
%%==============                                                       ==============%%
%%======                                  Section3                                   ======%%
%%====                                                                                        ====%%
%%==                                                                                              ==%%
%%====                    Global existence and boundedness                      ====%%
%%======                                                                                 ======%%
%%==============                                                       ==============%%
%%==============================================================%%

\section{Proof of Theorem \ref{mainthm}} \label{Sec3}

Thanks to Lemma \ref{bddu}, it is sufficient to derive 
an $L^p$-estimate for $u$ with some $p>\frac{n}{2}$. 
%
%%=================  Def. of test function f  ===================
 In order to establish the estimate for $u$ 
 we introduce the function 
 $f=f(x,t)$ by
    \begin{align}\label{Deff}
             f(x,t):= \exp\Big(
              -r\int_{\eta_1}^{v(x,t)}\chi(s)\,ds
              -\sigma \int_{\eta_2}^{w(x,t)} \xi(s)\,ds
                             \Big),
    \end{align}
%
%for the solution components $v, w$ satisfying  \eqref{LEw}, 
where $r, \sigma>0$ are some constants 
which will be fixed later. 
Here the function $f$ is finite valued, 
because integrability in \eqref{Deff} is assured  
%$\chi \in L^1(\eta_1, \infty)$, $\xi \in L^1(\eta_2,\infty)$ 
by \eqref{chiclass} and \eqref{xiclass} together with \eqref{LEw}. 
\vspace{2mm}
  
Then we give the following lemma which was proved 
in~\cite[Lemma~3.2]{CMY-2020} with $\mu=0$. 
Although in the literature we used the function $f$ with $\eta_1=\eta_2=0$, 
the conclusion of the following lemma does not depend on 
the choice of $(\eta_1, \eta_2)$.
%because, in \eqref{Deff}, 
%the difference between the case 
%$\eta_1=\eta_2=0$ and otherwise 
%is only multiplication by constants. 

%
%%=====================  Lemma 3.2  =======================
\begin{lem} \label{DIupf1}
 Let $r, \sigma>0$. 
 Then for all $p>1$, we have 
%
%%==========  DI of u^p f 1  ==========
    \begin{align}\label{est0}
            \frac{d}{dt}\int_\Omega u^pf
        &= I_1+I_2+I_3%+\mu p\int_\Omega u^p(1-u)f
            -r\int_\Omega u^pf\chi(v)(-v+u)
            -\sigma\int_\Omega u^pf\xi(w)(-w+u)
    \end{align}
 for all $t \in (0, \Tmax)$, where
%
%%==========  I_1, I_2, I_3  ==========
    \begin{align*}
             I_1 
        &:=  p\int_\Omega u^{p-1}f\nabla \cdot 
             \Big(\nabla u-u\chi(v)\nabla v+u\xi(w)\nabla w\Big),
    \\[1mm]
             I_2 
        &:=  -r\int_\Omega u^pf\chi(v)\Delta v,
    \\[1mm]
             I_3 
        &:=  -\sigma \int_\Omega u^p f\xi(w)\Delta w.
    \end{align*}
\end{lem}

We next state an estimate for $I_1+I_2+I_3$ in the following lemma. 
\begin{lem}
 Let $r, \sigma>0$, $\ep \in [0, 1)$ and put 
\begin{align*} 
\textsf{x} := u^{-1}|\nabla u|,\quad \textsf{y} := \chi(v)|\nabla v|,\quad
 \textsf{z} := \xi(w)|\nabla w|.
\end{align*} 
 Then for all $p>1$, the following estimate holds\/{\rm :}
    \begin{align}\label{est1}
               I_1+I_2+I_3
        &\le -\ep p(p-1) \int_\Omega u^pf\textsf{x}^2\notag\\
&\quad\,+\int_\Omega u^pf \cdot
            (a_1(\ep) \textsf{x}^2+a_2 \textsf{xy}+a_3 \textsf{xz}+a_4 \textsf{y}^2+a_5 \textsf{yz}+a_6 \textsf{z}^2)
    \end{align}
 for all $t \in (0, \Tmax)$, where
%
%%==========  Def. of a_1--a_6  ==========
    \begin{alignat*}{3}
        a_1(\ep) &:= -(1-\ep)p(p-1), & 
             &\qquad & 
        a_2 &:= p(p+2r-1),
    \\ 
        a_3 &:= p(p+2\sigma-1), &
             &\qquad & 
        a_4 &:= -r(p+r+\alpha),
    \\
        a_5 &:= pr+p\sigma+2r \sigma, &
             &\qquad &
        a_6 &:= -\sigma(-p+\sigma+\beta).
    \end{alignat*}
\end{lem}

\begin{proof}
Noting that $-\ep p(p-1)+a_1(\ep)=-p(p-1)$ 
for all $\ep \in (0,1)$, 
we see that the estimate \eqref{est1} is almost all the same as that in the case $\ep=0$ except multiplication by constants   
and is proved in~\cite[p.~10]{CMY-2020}. 
\end{proof}

We next give the following lemma, which is useful to show that the second term 
on the right-hand side of \eqref{est1} is nonpositive.

\begin{lem}\label{A_1A_2}
 Assume that $\alpha, \beta$ satisfy \eqref{condiab}. 
 Then there exist $p>\frac{n}{2}$ and $r, \sigma>0$ 
 such that
    \begin{align} \label{Syl}
             A_1:=\left|
                    \begin{array}{cc}
                                   a_1(0) & \frac{a_3}{2}
                    \\
                        \frac{a_3}{2} & a_6
                    \end{array}
                    \right|>0
    \quad
             {\rm and}
    \quad
             A_2:=\left|
                    \begin{array}{ccc}
                                   a_1(0) & \frac{a_3}{2} & \frac{a_2}{2}
                    \\
                        \frac{a_3}{2} &            a_6 & \frac{a_5}{2}
                    \\
                        \frac{a_2}{2} & \frac{a_5}{2} &            a_4
                    \end{array}
                    \right| <0.
    \end{align}
\end{lem}

\begin{remark}
% We have shown that there exist $p>\frac{n}{2}$ and $r, \sigma>0$ 
%such that $A_1>0$, $A_2 \le 0$ (see~\cite[proof of Lemma~3.3]{CMY-2020}). 
%More precisely, we set
%    \begin{align*}% \label{phi2}
%        \varphi_2(r):=c_1r^2+c_2r+c_3
%    \end{align*}
%% 
% with some $c_1, c_2, c_3>0$ and showed that there exists $r>0$ such that 
% $\varphi_2(r) \le 0$. 
% At that time, we paid attention to the convexity of 
% the quadratic function $\varphi_2$, and guaranteed existence of 
% $r>0$ such that $\varphi_2(r) \le 0$ 
% by confirming the following two conditions:
%%
%%%========== condi. of r ==========
%    \begin{align}
%         r_0>0, \quad &\mbox{where 
%         $r_0$ is the axis of the parabola $\varphi_2$,}\label{condir_0}\\
%         D_2>0, \quad &\mbox{where $D_2$ 
%         is the discriminant of $\varphi_2$.}\label{condiD_2}
%    \end{align}
%However, we can derive $\varphi_2(r)<0$ for some $r>0$ 
%by the above conditions. 
%After that, we confirmed that ``$A_2\ge0$'' could be replaced by 
%``$A_2>0$''. 
In~\cite[Proof of Lemma~3.3]{CMY-2020}, we showed that 
there exist $p>\frac{n}{2}$ and $r, \sigma>0$ such that 
$A_1>0$, $A_2 \le 0$. 
Here, $A_2 \le 0$ can be refined as $A_2<0$ 
for some $p>\frac{n}{2}$ and $r, \sigma>0$. 
More precisely, in the literature we set
    \begin{align*}% \label{phi2}
        \varphi_2(r):=c_1r^2+c_2r+c_3
    \end{align*}
 with some $c_1, c_2, c_3>0$ and find $r>0$ such that 
 $\varphi_2(r) \le 0$ by the following two conditions:
%
%%========== condi. of r ==========
    \begin{align*}
         r_0>0, \quad &\mbox{where 
         $r_0$ is the axis of the parabola $\varphi_2$,}%\label{condir_0}
       \\
         D_2>0, \quad &\mbox{where $D_2$ 
         is the discriminant of $\varphi_2$.}%\label{condiD_2}
    \end{align*}
These conditions, indeed, show that 
$\varphi_2(r)<0$ for some $r>0$ which implies to $A_2<0$. 
\end{remark}

Combining the above three lemmas, 
we can derive the following important inequality 
which leads to the $L^p$-estimate for $u$. 
%
%%=====================  Lemma 3.3  =======================
\begin{lem} \label{DIupf2}
 Assume that $\chi, \xi$ satisfy \eqref{chiclass}--\eqref{xiineq} with 
 $\alpha, \beta$ which fulfill \eqref{condiab}. 
 Then there exist $p>\frac{n}{2}$ and $r, \sigma>0$ 
 such that
%
%%==========  I_1+I_2+I_3 \le 0  ==========
    \begin{align*}
            \frac{d}{dt}\int_\Omega u^pf
            +\ep_0 p(p-1)\int_\Omega u^{p-2}f|\nabla u|^2
        &\le
            -r\int_\Omega u^pf\chi(v)(-v+u)
            -\sigma\int_\Omega u^pf\xi(w)(-w+u)
    \end{align*}
for all $t \in (0, \Tmax)$ with some $\ep_0 \in (0,1)$.
\end{lem}
%
%%==================  Proof of Lemma 3.3  ====================
%
%%==================  Proof of Lemma 3.3  ====================
\begin{proof}
We put
%
%%==========  det condi.  ==========
    \begin{align*}% \label{Syl2}
             A_1(\ep):=\left|
                    \begin{array}{cc}
                                   a_1(\ep) & \frac{a_3}{2}
                    \\
                        \frac{a_3}{2} & a_6
                    \end{array}
                    \right|
    \quad
             {\rm and}
    \quad
             A_2(\ep):=\left|
                    \begin{array}{ccc}
                                   a_1(\ep) & \frac{a_3}{2} & \frac{a_2}{2}
                    \\
                        \frac{a_3}{2} &            a_6 & \frac{a_5}{2}
                    \\
                        \frac{a_2}{2} & \frac{a_5}{2} &            a_4
                    \end{array}
                    \right|
    \end{align*}
 for $\ep \in [0, 1)$. 
 Since $A_1(0)>0$ and $A_2(0)<0$ hold in view of \eqref{Syl} and 
 the function $a_1 \colon \ep \mapsto -(1-\ep)p(p-1)$ is continuous at 
 $\ep=0$, 
 we can find $\ep_0 \in (0,1)$ such that $A_1(\ep_0)>0$ and $A_2(\ep_0)<0$. 
 By using the Sylvester criterion, we have 
\begin{align}\label{est2}
a_1(\ep_0) \textsf{x}^2+a_2 \textsf{xy}+a_3 \textsf{xz}+a_4 \textsf{y}^2
+a_5 \textsf{yz}+a_6 \textsf{z}^2 \le 0.
\end{align}
Combining \eqref{est1} and \eqref{est2} with \eqref{est0}, 
we arrive at the conclusion.
\end{proof}

We now show the desired $L^p$-estimate for $u$ 
with some $p>\frac{n}{2}$.
%
%%=====================  Lemma 3.4  =======================
\begin{lem} \label{uLpbdd2}
 Let $p>\frac{n}{2}$. 
 Assume that $\chi, \xi$ satisfy \eqref{chiclass}--\eqref{xiineq} 
 with $\alpha, \beta$ which fulfill \eqref{condiab}. 
 Then there exists $C>0$ such that 
%
%%==========  uLp  ==========
    \begin{align*}
             \|u(\cdot, t)\|_{L^p(\Omega)} \le C
%    \quad 
%             {\it for\ all}\ t \in (0, \Tmax).
    \end{align*}
 for all $t \in (0, \Tmax)$.
\end{lem}
%
%%==================  Proof of Lemma 3.4  ====================
\begin{proof}
 By Lemma \ref{DIupf2}, we see from the 
 positivity of $\chi, \xi$ and \eqref{chibound}, \eqref{xibound} that 
%
%%==========  \int u^pf  ==========
\begin{align} \label{L34}
               \frac{d}{dt}\int_\Omega u^pf+\ep_0 p(p-1)\int_\Omega u^{p-2}f|\nabla u|^2
        &\le 
               -r\int_\Omega u^pf\chi(v)(-v+u)
               -\sigma\int_\Omega u^pf\xi(w)(-w+u)
    \notag\\
        &\le r\chi_0\int_\Omega u^pf+\sigma\xi_0\int_\Omega u^pf
    \notag\\
        &= (r\chi_0+\sigma \xi_0)\int_\Omega u^pf
    %\notag
    \end{align}
for all $t \in (0, \Tmax)$ with some $\ep_0 \in (0,1)$. 
Noting $f \le 1$ in view of \eqref{Deff} and then 
using the Gagliardo--Nirenberg inequality together with
the mass conservation property \eqref{mcp}, we have 
%
%%==========  \int u^pf  ==========
    \begin{align}\label{GNest}
           \int_\Omega u^p f
    &\le \int_\Omega u^p=\|u^{\frac{p}{2}}\|_{L^2(\Omega)}^2\notag\\
    &\le c_1\Big(\|\nabla u^{\frac{p}{2}}\|_{L^2(\Omega)}
           +\|u^{\frac{p}{2}}\|_{L^{\frac{2}{p}}(\Omega)}\Big)^{2\theta}
           \|u^{\frac{p}{2}}\|_{L^{\frac{2}{p}}(\Omega)}^{2(1-\theta)}\notag\\
    &= c_1\Big(\|\nabla u^{\frac{p}{2}}\|_{L^2(\Omega)}
           +\|u_0\|_{L^1(\Omega)}^{\frac{p}{2}}\Big)^{2\theta}
           \|u_0\|_{L^1(\Omega)}^{p(1-\theta)}\notag\\
    &\le c_2\|\nabla u^{\frac{p}{2}}\|_{L^2(\Omega)}^{2\theta}
           +c_3
    \end{align}
 for all $t \in (0, \Tmax)$ with some $c_1, c_2, c_3>0$, where 
 $\theta:=\frac{\frac{pn}{2}-\frac{n}{2}}{\frac{pn}{2}+1-\frac{n}{2}} \in (0,1)$. 
 Also, noticing from $\chi \in L^1(\eta_1, \infty)$, $\xi \in L^1(\eta_2,\infty)$
 (see \eqref{chiclass}, \eqref{xiclass}) that 
%
%%==========  boundedness of f  ==========
    \begin{align}\label{fbdd}
             f \ge c_4
              :=\exp\Big(-r\int_{\eta_1}^\infty \chi(s)\,ds
                              -\sigma\int_{\eta_2}^\infty \xi(s)\,ds
                        \Big)>0
    \quad
             {\rm on}\ \Omega \times (0, \Tmax),
    \end{align}
 we obtain 
\begin{align}\label{up/2est}
\frac{4c_4}{p^2}\|\nabla u^{\frac{p}{2}}\|_{L^2(\Omega)}^{2}=\frac{4c_4}{p^2}\int_\Omega |\nabla u^{\frac{p}{2}}|^2
\le \int_\Omega u^{p-2}f|\nabla u|^2.
\end{align}
for all $t \in (0, \Tmax)$. 
Combining \eqref{GNest}, \eqref{up/2est} with \eqref{L34}, we see that
%
%%==========  \int u^pf 2  ==========
\begin{align*}% \label{intupf2}
               \frac{d}{dt}\int_\Omega u^pf
        &\le c_5\int_\Omega u^pf
               -c_6\Big(\int_\Omega u^p f\Big)^{\frac{1}{\theta}}+c_7
    \end{align*}
for all $t \in (0, \Tmax)$  with some $c_5, c_6, c_7>0$. 
 This provides 
 a constant $c_8>0$ such that 
%
%%==========  \int u^pf 3  ==========
    \begin{align*}
             \int_\Omega u^pf \le c_8,
    \end{align*}
which again by \eqref{fbdd} implies
%
%%==========  \int u^p  ==========
    \begin{align*}
             \int_\Omega u^p \le \frac{p^2c_8}{4c_4}
    \end{align*}
for all $t \in (0, \Tmax)$ 
 and thereby precisely arrive at the conclusion.
\end{proof}
\vspace{5mm}

 We are in a position to complete the proof of Theorem \ref{mainthm}.
 If $\chi, \xi$ satisfy \eqref{chiclass}--\eqref{xiineq} 
 with $\alpha, \beta$ fulfilling \eqref{condiab}, 
 then, according to the relations that $\chi(s) \le \chi(\eta_1)$ for all $s \ge \eta_1$ and $\xi(s) \le \xi(\eta_2)$ for all $s \ge \eta_2$
 (see \eqref{chiineq} and \eqref{xiineq}), 
 a combination of 
 Lemmas~\ref{bddu} and~\ref{uLpbdd2}, along 
 with \eqref{BU}, leads to the end of the proof.

%
%==========================================================
%%%%%%%                                             %%%%%%%
  %%%                                                 %%%
 %%%                                                   %%%
%%%                    Acknowledgments                  %%%
 %%%                                                   %%%
  %%%                                                 %%%
%%%%%%%                                             %%%%%%%
%==========================================================
%\smallskip

%%==============================================================%%
%%==============                                  ==============%%
%%======                                                  ======%%
%%====                                                      ====%%
%%==                         Reference                        ==%%
%%======                                                  ======%%
%%==============                                  ==============%%
%%==============================================================%%

%\bibliography{ref}
%\bibliographystyle{plain}
\end{document}